\documentclass[12pt,leqno,twoside]{amsart}

\usepackage[latin1]{inputenc}
\usepackage[T1]{fontenc}
\usepackage[colorlinks=true, pdfstartview=FitV, linkcolor=blue, citecolor=blue, urlcolor=blue]{hyperref}
\usepackage{amstext,amsmath,amscd, bezier,indentfirst,amsthm,amsgen,enumerate, geometry}
\usepackage[all,knot,arc,import,poly]{xy}
\usepackage{amsfonts,color}
\usepackage{amssymb}
\usepackage{latexsym}
\usepackage{epsfig}
\usepackage{graphicx}
\usepackage{srcltx}
\usepackage{enumitem}

\newcommand{\fin}{\hspace*{\fill}$\square$\vspace*{2mm}}

\theoremstyle{plain}
\newtheorem{theorem}{Theorem}[section]
\newtheorem{lemma}[theorem]{Lemma}

\newtheorem{corollary}[theorem]{Corollary}

\theoremstyle{definition}
\newtheorem{definition}[theorem]{Definition}
\theoremstyle{remark}
\newtheorem{remark}[theorem]{\sc Remark}
\newtheorem{example}[theorem]{\sc Example}

\def\bC{{\mathbb C}}

\def\bN{{\mathbb N}}
\def\bP{{\mathbb P}}

\def\bR{{\mathbb R}}

\def\bX{{\mathbb X}}

\def\cA{{\mathcal A}}
\def\cB{{\mathcal B}}
\def\cC{{\mathcal C}}

\def\cM{{\mathcal M}}

\def\ity{\infty}
\def\e{\varepsilon}

\def\max{{\rm max}}
\def\dist{{\rm dist}}

\def\Sing{{\rm Sing}}

\def\const.{{\rm const.}}

\def\im{{\rm Im}}
\def\m{{\setminus}}

\begin{document}
\title[Atypical points at infinity]{Atypical points at infinity and algorithmic detection of the bifurcation locus of real polynomials}

\author{Luis Renato G. Dias}
\address{Faculdade de Matem\'atica, Universidade Federal de Uberl\^andia, Av. Jo\~ao Naves de \'Avila 2121, 1F-153 - CEP: 38408-100, Uberl\^andia, Brazil.}
\email{lrgdias@famat.ufu.br}

\author{\sc Cezar Joi\c{t}a}
\address{Institute of Mathematics of the Romanian Academy, P.O. Box 1-764,
 014700 Bucharest, Romania.} 
\email{Cezar.Joita@imar.ro}

\author{Mihai Tib\u{a}r}
\address{Univ. Lille, CNRS, UMR 8524 -- Laboratoire Paul Painlev\'e, F-59000 Lille,
France}  
\email{mihai-marius.tibar@univ-lille.fr}



\thanks{LRGD and MT acknowledge partial support by the Math-AmSud grant 18-MATH-08. LRGD acknowledges the CNPq-Brazil grant 401251/2016-0 and 304163/2017-1.  CJ acknowledges the CNCS grant PN-III-P4-ID-PCE-2016-0330.  
The authors  acknowledge  support by the Labex CEMPI (ANR-11-LABX-0007-01).}

\begin{abstract} 
We show that the variation of the topology at  infinity of a two-variable polynomial function 
is localisable at a finite number of  ``atypical points'' at infinity.   We construct an effective algorithm with low complexity in order to detect sharply the bifurcation values of the polynomial function. 
\end{abstract}

\maketitle

\section{Introduction} 

 Let $f:\bR^2 \to\bR$ be a non-constant polynomial function.   One says that $\lambda \in\bR$ is a \emph{typical value} of $f$ if  there exists a disk $D\subset \bR$ centred at $\lambda$ such that the restriction $f_| : f^{-1}(D) \to D$ is a C$^{\infty}$ trivial fibration. If the contrary occurs, then we say that $\lambda$ is a \emph{bifurcation value} (or a \emph{atypical value}) of $f$, and then its corresponding fibre $f^{-1}(\lambda)$ is called atypical.  A critical value $c\in f(\Sing f)$ is a bifurcation value, since one cannot have a  C$^{\infty}$ trivial fibration at $c$.\footnote{If instead of C$^{\infty}$ we ask only ``C$^0$  trivial fibration'' in the definition of typical value, then simple examples like   $f(x,y) = x^{3}+y^{2}$  show that a singular fibre may be typical (the fibre $f^{-1}(0)$ in this case).} Nevertheless, well-known examples like $f(x,y) = x + x^{2}y$ show that non-singular fibres may also be atypical (the fibre $f^{-1}(0)$ in this case).
 
 The  \emph{bifurcation set}, denoted by $\cB_{f}$, is therefore the union of the finite set of singular values $f(\Sing f)$ with the set of   regular values which are atypical.

  In  the late 1990's, the study of the atypical values led to the discovery in \cite{TZ} of two possible phenomena which may occur at infinity: the \emph{vanishing} and the  \emph{splitting} at infinity.
These play a crucial role in the detection of the bifurcation set.   
  Subsequent  papers gave further methods to detect the vanishing and the splitting phenomena, either in large generality \cite{CP}, or for Newton nondegenerate polynomials \cite{INP},  or for rational functions on algebraic surfaces in $\bR^n$, see \cite{HN}. 
  An extension to  maps $f:\bR^n \to\bR^{n-1}$,  $n\ge 2$,  has been produced in  \cite{JoT}. 
 All these results are based on the fact that the affine fibres of $f$ are of dimension 1, that their connected components can be distinguished (algorithmically, see e.g. \cite{Sc}),  and that one can detect the vanishing and the splitting phenomena on each of them. For this later detection there are yet no finite algorithms\footnote{In \cite{CP}, the authors do not use the notion of ``splitting'' but something else called ``cleaving''. They  use parametrisations of affine curves without any truncation (whereas examples are worked out with some truncations), and effectivity is no concern either.}. It is thus still a problem to find bounded  algorithms and effective procedures for reaching solutions. 

  In case of polynomial functions of $n>2$ variables or,  more generally, for polynomial maps 
 $\bR^n \to\bR^{p}$ with $p\ge 1$, one can only estimate the set of the atypical values but not determine it precisely: \cite{JK, DT, DTT} by using a finite dimensional arc space, and \cite{JeT} using polar curves.
 Effective algorithmic procedures have been given in \cite{JK, DT, DTT}. Applications of the bifurcation set detection to the optimisation of polynomial mappings have been given recently, for instance in \cite{BS, KPT, Sa}.
  
 \smallskip 
 
 The new issue of our study is the effective algorithmic exact detection of the bifurcation  values. 
 
This is based on the localisation at infinity of the phenomena that produce the atypical values, which is  a new approach in the real setting (see also \cite{INP} for a  certain class of polynomials treated with toric methods), whereas in the complex setting this has been  achieved since more than 25 years ago,  under the condition that ``singularities at infinity are isolated'', see  \cite{ST, Pa1, Ti-compo} \cite[Section 2]{Ti-book}. More precisely,    the top homology group (with integer coefficients) of the generic fibre may have cycles that disappear at infinity when tending to the fibre over $\lambda$. If some cycles are vanishing at certain points at infinity of this fibre over $\lambda$, then those points are called \emph{atypical points at infinity}. The following localisation result for polynomials of two complex variables polynomials holds: 
 
 \smallskip 
\noindent  \textbf{Theorem} \cite{ST, Pa1, Ti-book}\footnote{The papers \cite{ST, Pa1, Ti-book} present more general statements in $n\ge 2$ complex variables, all of which amount to the same in case  $n=2$.}
\emph{Let $P: \bC^{2}\to \bC$ be a non-constant polynomial. A regular value $\lambda$ is atypical if and only if the fibre $P^{-1}(\lambda)$ has atypical points at infinity}.

 \medskip 
  We introduce and develop  here its real counterpart.
We denote  by $L^{\ity} := \{z=0\} \simeq \bP^{1}$ the line at infinity of the embedding $\bR^{2}\subset \bP^{2}$, we
 introduce (Definition \ref{d:atyp}) a subset $\cA^{\ity} \subset L^{\ity}\times \bR$ of \emph{atypical points at infinity} of $f$, which turns out to be a finite set (Theorem \ref{t:atyppoint}), and we show:

 \begin{theorem}[Localisation at infinity]\label{t:localis}
 A regular value $\lambda$ is a bifurcation value of $f$ if and only if there exists $p\in L^{\ity}$ such that
  $(p, \lambda)\in L^{\ity}\times \bR$ is  an atypical point at infinity.
\end{theorem}
 
We will actually prove a more general statement -- Theorem \ref{t:main2} -- in which we drop the regularity of the
fibre $f^{-1}(\lambda)$ and  allow $\Sing f^{-1}(\lambda)$  to be a compact set, and in exchange we replace ``bifurcation value'' with  the notion of ``atypical values at infinity'' to be defined just below.  Our proofs then show that the statement of Theorem \ref{t:main2} holds for any restriction $F = f_{| \bR^{2}\m D}$ where $D$ is some disk. 

So,  in order to focus on the problems that occur at infinity and to separate them from the behaviour in the rest of the space\footnote{Let us remark here that the sets  $f(\Sing f)$ and $\cA_{f}$ may be disjoint, see Example \ref{ex:1}.}, it is natural to  introduce\footnote{See also \cite[Definition 1.1.5]{Ti-book}, \cite[Theorem 4.7]{JoT}, \cite[Definition 2.1]{DT}, \cite[Definition 2.1]{DTT}, \cite[Page 1]{INP}, \cite[Page 1]{JK}.} the set of  \emph{atypical values at infinity}:

  \begin{definition}\label{d:at-ity}
		Let $f:\bR^2 \to\bR$ be a polynomial function. We say that $\lambda$ is a \emph{typical value at infinity} of $f$ if  there exists an interval $I := ]\lambda- \eta , \lambda +\eta[ \subset \bR$ centred at $\lambda$ and a compact set $K\subset\bR^2$ such that the restriction $f_| : f^{-1}(I) \setminus K \to I$ is a  C$^{\infty}$   trivial fibration. If the contrary occurs, then we say that $\lambda$ is an \emph{atypical value at infinity} of $f$, and then its corresponding fibre $f^{-1}(\lambda)$ is also called \emph{atypical at infinity}. 
		
		We denote by $\cA_{f}$ the atypical values at infinity. 
	\end{definition}

  The proof of Theorem \ref{t:localis} via Theorem \ref{t:main2} relies on the extension, in the new setting of Definition \ref{d:at-ity},   of the characterisation  of regular bifurcation values from \cite[Corollary 4.7]{JoT} in terms of vanishing and splitting as defined in Definition \ref{d:vanish}  (Theorem \ref{t:vansplit}).   
  
  We then show that our detection problem reduces to the computation of the
  local topological type of a real algebraic curve and its variation in a one-parameter family.
  We explain how this compares to King's approach to the topological triviality of a family of function germs via his notion of \emph{coalescing} \cite{Ki}. 
 
In some more detail, our reduction can be summarised as follows. 
By using the Euclidean distance function, we define a finite subset $\cM$ of points $(p, \lambda)\in \bP^{1}\times \bR$ containing the set $\cA^{\ity}$ of atypical points at infinity.  We use Puiseux expansions truncated at a certain level in order to embed $\cM$ into a finite set which can be reached in finitely many steps, \S \ref{ss:M}.
Then we construct a sharp algorithmic test to be applied at each of the points of this finite set  in order to select precisely the atypical points, see \S \ref{ss:disk} and \S \ref{ss:algovanish}. The algorithms are concentrated in \S \ref{s:algo}.



 \section{Localisation of the variation of topology at infinity} 
 
Let us show that the bifurcation values can be detected by the local behaviour at certain points at infinity.
 
We recall here some standard notations that we shall use. We denote by $x, y$ the coordinates of the affine space $\bR^2$ and by $[x:y:z]$ a point in the projective space $\bP^2$, where $\bR^2$ is identified with the open set $\{[x:y:z] \in \bP^2 \mid z \neq 0 \}$ of $\bP^2$ through the map  $(x,y) \to [x: y: 1]$, and where $L^{\infty} := \{[x:y:z] \in\bP^2 \mid z=0 \} \simeq \bP^{1}$ will denote the \emph{line at infinity}.  

Let $\bX :=\{([x:y:z],t) \in\bP^2\times \bR \mid \tilde{f}(x,y,z) -tz^d = 0\}$, where $\tilde{f}$ denotes the homogenization of $f$ in the variables $z$. \footnote{The space $\bX$ is not necessarily the closure in $\bP^2\times \bR$ of the graph of $f$ like it is the case over $\bC$, see \cite[Note 1.1.3]{Ti-book}, for instance in the example $f(x,y)=x^4 + y^2$.}

Let $\bX^{\infty}:=\bX\cap (L^{\infty}\times\bR)$. We denote by $\sigma: \bX\to \bP^2$ and $\tau: \bX\to \bR$ the canonical projections, respectively. Since $\sigma(\bX^{\infty}) = \{[x:y:0] \in L^{\infty} \mid f_d(x,y)=0  \}$, it follows that $\sigma(\bX^{\infty})$ is a finite set.  
We identify $\bR^{2}$ with its embedding into the graph of $f$ followed by the inclusion of this graph into $\bX$,  namely $(x,y) \mapsto (x, y, f(x,y)) \mapsto ([x: y: 1], f(x,y))\in \bX$.  For instance a ball $B_{R}\subset \bR^{2}$ will be viewed as a subset of $\bX$ via this identification.

\smallskip

In the following we focus on  a fibre $f^{-1}(\lambda)\subset \bR^{2}$ with at most  isolated singularities (or, more generally, with compact singular locus) and study it in a neighbourhood of infinity.

\subsection{Atypical points at infinity}\label{ss:atyp}

Without loss of generality, let $p=[0:1:0]\in L^{\ity}$, and consider a local chart $U \simeq\bR^{2}\subset \bP^{2}$ with origin at $p$. Let $g_t: (\bR^{2},0) \to (\bR, 0)$,   $g_t(x,z) :=\tilde{f}(x,1,z) -t z^d$. 
This is a local family of polynomial function germs $g_t$ at the point $p$, 
also defining a family of algebraic curve germs $C_{t}:= \{g_t =0\}$ at $(0,0)\in \bR^2$,  
 for $t$ close enough to $\lambda$,  from the right as well as from the left.   All these data depend on the coordinates and on the local chart $\bR^2$ at infinity.

 The notation $t\to \lambda$ will mean that $t$ tends to $\lambda$ \emph{either from the right or from the left}. 
 

\begin{definition}\label{d:split-van}
Let $p\in L^{\ity}$.	We say that $f$ has a \emph{splitting at $(p,\lambda)\in L^{\ity}\times \bR$}, shortly (S$_{(p,\lambda)}$),  if there is a small disk $D_{\e}$ at $p\in U \simeq \bR^{2}$ such that the representative of 
	the curve $C_{t}$ in $D_{\e}$ has a connected component $C_{t}^{i}$ such that
	$C_{t}^{i} \cap \partial D_{\e} \not= \emptyset$  for all $t< \lambda$ (or for all $t> \lambda$) close 
	enough to  $\lambda$, and that the Euclidean distance $\dist(C_{t}^{i}, p) \not=0$ and tends to 0 when $t\to \lambda$.
	
We say that $f$ has a \emph{vanishing loop at $(p,\lambda)\in L^{\ity}\times \bR$}, shortly (V$_{(p,\lambda)}$), if there is a small disk $D_{\e}$ at $p\in \bR^{2}$ such that  $C_{t}\cap D_{\e}\m \{p\}$ has a non-empty connected component $C_{t}^{i}\m \{p\}$ with $C_{t}^{i} \cap \partial D_{\e}= \emptyset$ for all $t< \lambda$ (or for all $t> \lambda$) close 
	enough to  $\lambda$, such that
 $\lim_{t\to\lambda} C_t^i  \cap D_{\e} = \{p\}$.
\end{definition}
\begin{definition}[Atypical point at infinity]\label{d:atyp}
We say that $(p, \lambda)\in L^{\ity}\times \bR$ is an \emph{atypical point at infinity} of $f$ if there is (S$_{(p,\lambda)}$) or (V$_{(p,\lambda)}$).  We denote by $\cA^{\ity}$ the set of atypical points at infinity of $f$.
\end{definition}

The following set turns out to play the key role in the detection of the atypical points at infinity:
 	\begin{definition}[Milnor set at infinity]\label{d:milnor}
 		Let $f\colon\bR^2\to\bR$ be a polynomial function and let $\rho:\bR^2\to\bR$, $\rho(x,y)=x^2+y^2$, be the square of the Euclidean distance from the origin. The \emph{Milnor set of $f$}, denoted by $M(f)$,  is the critical locus of the mapping $(f, \rho)\colon\bR^2\to\bR^{2}$. 
 \end{definition}

\begin{remark}
 	The Milnor set is defined as $M(f)= \{ (x,y) \in\bR^2 \mid h(x,y) = 0 \}$, where $h$ is the polynomial function $h(x,y) = y\frac{\partial f}{\partial x}(x,y) - x\frac{\partial f}{\partial y}(x,y)$. Thus  $\dim M(f)=2$ if   and only if $h\equiv 0$ (see for instance \cite[page 10]{Mi}).  Moreover, $\dim M(f) =2$ if and only if   $f = P \circ \rho$ for some polynomial $P$ of one real variable. In this case the fibres of $f$ are compact (unions of circles) or empty,  and we have $\cB_{f} = f(\Sing f)$. 
\end{remark}

	From now on we shall assume that $f$ is \emph{primitive} in the strict sense that $f \not= P \circ \rho$, where $\rho$ denotes the square of the Euclidean distance function.  
In this case one has $\dim M(f) = 1$.  
	
Nevertheless, the Milnor set may intersect  some fibres along sets of dimension 1, as the following example shows.  
\begin{example}\label{ex:1}
 $f(x,y) = (x^2+ y^2 - 1)(x^2 + 1 )$,where $\Sing f =\{ (0,0)\}$ and $f(\Sing f)=-1$, and $M(f)=\{(x,y) \mid 2xy(x^2+y^2 - 1)=0 \}$.  We have that  $f^{-1}(0)=\{ (x,y) \mid x^2 + y^2 -1 =0 \}$ is a regular fibre but $f^{-1} (0) \cap M(f) = f^{-1}(0)$.
 Here we have $\cA_{f} =\emptyset$ and $\cB_{f} = \{ -1\}$.
\end{example}


\subsection{Proving Theorem \ref{t:localis}} \label{ss:proofmain}\ 

We  enlarge the setting and prove here several results for the restriction of the polynomial function $f$ to the complement of some compact set (which may be empty, and then we get just $f$). 

To be more precise, we consider restrictions of the type $F:= f_{|}: \bR^{2}\m \overline{D_R} \to \bR$ for some $R\ge 0$, where $D_{R}$ denotes the open disk at the origin of radius $R$.
A non-singular fibre $X_t := F^{-1}(t)$ of $F$ is a disjoint union of one dimensional smooth manifolds  $X_t=\sqcup_{j}X_t^{j}$. Each connected component, denoted  $X_t^{j}$,  is diffeomorphic either to the circle $S^1$ (and we will call it ``circle'') or to $\bR$ (and we will call it ``line''). 

Theorem \ref{t:localis} will follow from the following more general statement, under the above notations, and of the  notations of Definitions \ref{d:at-ity} and \ref{d:atyp}.
\begin{theorem}\label{t:main2}
 Let  $f^{-1}(\lambda)$ be a fibre with a compact singular set. Then  $\lambda \in \cA_{f}$
if and only if there exists $p\in L^{\ity}$ such that  $(p, \lambda) \in \cA^{\ity}$.
\end{theorem}


Under the above notations,
let us state the definition of splitting and vanishing conditions, extending\footnote{One may compare our definition of vanishing at $\lambda$   to the contrary of  ``no vanishing at $\lambda$'' as stated in \cite{TZ}, \cite[Definition 3.1]{JoT}. Similarly, the precise contrary of what we call here splitting at $\lambda$  compares to the  notion of  ``strong non-splitting'' of \cite[Definition 4.3]{JoT}. The first part of the paper \cite{JoT}, as well as the whole paper  \cite{TZ}, use a slightly weaker notion of ``splitting'' than in  our Definition \ref{d:vanish}. The difference is that the  non-splitting of \cite{TZ} did not include the behaviour of compact components whereas our Definition \ref{d:vanish} does. See \cite[Example 3.2]{TZ} where a circle component ``splits'' into a line component.} those in \cite{TZ, JoT}. For the notion of limit of sets we refer to \cite[Definition 4.1]{JoT} and to the relevant articles cited in  \cite{JoT}.  More explicitly,  let $\{M_t\}_{t\in R}$ be a family of subsets of $\bR^{m}$.
Then $\lim_{t\to\lambda} M_t$ is by definition the set of all points $x\in\bR^{m}$
such that there exists a sequence  $t_k\in\bR$ such that $t_k\to\lambda$ (from one side, as the general convention in this paper) and a sequence of points  $x_k\in M_{t_k}$ such that $x_k\to x$.

\begin{definition}\label{d:vanish}
Let $\lambda \in \bR $  such that $\Sing f^{-1}(\lambda)$ is a compact set.

\noindent
One says that $f$ \emph{has a vanishing component at infinity at $\lambda$}, if
 $\lim_{t\to\lambda}\max_{j}\mathrm{inf}_{q\in X_t^{j}}\Vert q \Vert =\infty$.
 For short, we denote this by (V$_{\lambda}$). 

\smallskip
\noindent
 One says that \emph{$f$ is splitting at infinity at $\lambda$} if there is $\eta>0$ and  a continuous family of analytic  paths $\gamma_t:[0,1]\to X_t$ for $t\in ]\lambda-\eta, \lambda[$ or  for $t\in ]\lambda, \lambda+\eta[$,  such that: \\
 (1).  $\im \gamma_{t} \cap M(f) \not= \emptyset$, with  $\lim_{t\to \lambda}\| q_{t}\| = \infty$ for any 
 $ q_{t} \in \im \gamma_{t} \cap M(f)$, and \\
 (2). the limit set  $\lim_{t\to \lambda} \im \gamma_t$ is not connected.\footnote{Equivalently: (2'). there exist $j$ and  a continuous family of analytic  paths $\gamma_t:[0,1]\to X^{j}_t$ such that  the limit set  $\lim_{t\to \lambda} \im \gamma_t$ is not connected.}\\
For short, we denote this by (S$_{\lambda}$). 		 
\end{definition}
Let us remark that \label{l:vanishcompact}
   only non-compact connected components $X_t^{j}$ may be vanishing at infinity. Also, remark that we get equivalent notions if in this definition  we replace   $f$ by $F$ for any $R>0$, since only the behaviour of the fibres at infinity counts.

\begin{theorem} \label{t:vansplit}
 Let  $f^{-1}(\lambda)$ be a fibre with a compact singular set.
 
Then the value $\lambda$ of $f$ is atypical at infinity  if and only if  (V$_{\lambda}$) or (S$_{\lambda}$).
\end{theorem}

The statement of Theorem \ref{t:vansplit} can be viewed as an extension of \cite[Corollary 4.7, Theorem 4.8]{JoT} and  of \cite[Theorem  2.5]{TZ} in the spirit of our new Definition \ref{d:vanish}.  With respect to \cite{TZ}, our new Definition \ref{d:vanish} of (S$_{\lambda}$) encompasses, like the one in  \cite{JoT}, a more general splitting phenomenon, including the case when a circle splits at infinity into a line, like  in  \cite[Example 3.2]{TZ}. 

\begin{proof}[Proof of Theorem \ref{t:vansplit}]
From the definitions it follows that the existence of either  (V$_{\lambda}$) or (S$_{\lambda}$) implies that $\lambda$ is an atypical value, so this implication is trivial.  Let us show the converse.

We first need  to define a large enough disk outside where to concentrate our study according to Definition \ref{d:at-ity}.
We use the following result which can be derived from the Tarski-Seidenberg principle (see e.g. \cite[pag. 35]{Ti-book}):	
\begin{lemma}\label{l:disk}
 Let $f:\bR^2 \to \bR $ be a polynomial function, and let $f^{-1}(\lambda)$ be a fibre with a compact set of singular points.  Then there is $R\gg 1$  such that $f^{-1}(\lambda) \cap M(f) \setminus D_{R} =\emptyset$. \fin
\end{lemma}   
We may then invoke the following result:

\smallskip

\noindent
  \cite[Corollary 4.7]{JoT}: \emph{A non-singular fibre $f^{-1}(\lambda)$ of a polynomial function   $f:\bR^2 \to\bR$ is atypical if and only if 
 there are either vanishing or splitting components at $\lambda$,}
 
 \medskip
 
 \noindent and apply it here,  instead of $f$ itself,  to the restriction $F = f_{|\bR^{2}\m \overline{D_R}}$ with $R$ as in Lemma \ref{l:disk}. 
 The proof in case of the restriction $F$  follows the one in \cite{JoT} faithfully and consist of a stepwise  reduction  to the situation of a family of non-compact connected components of  the fibres of $F$ over a small enough interval $I := ]\lambda- \eta , \lambda +\eta[$ such that $f^{-1}(I) \cap M(f) \cap \partial D_{R} = \emptyset$, i.e. where the fibres of $F$ are transversal to the boundary circle $\partial D_{R}$.
 
In order to conclude, we may then apply the following result from  \cite{TZ} which is in turn a particular case of Palmeira's theorem \cite{Pa}:
 \smallskip
 
 \noindent
  \cite[Proposition 2.7]{TZ}: 
\emph{Let $M \subseteq \bR^2$ be a smooth submanifold of 
dimension $2$ and let $g : M \to \bR$ be a smooth function without singularities and such that all its fibres $g^{-1}(t)$ are closed in $M$ and diffeomorphic to $\bR$.  Then $g$ is a {\rm C}$^{\infty}$ trivial fibration.}

\medskip

This ends the proof of Theorem \ref{t:vansplit}.
\end{proof}

\subsection{End of the proof of Theorem \ref{t:main2}}\label{ss:milnorset}\

Let $\overline{M(f)} \subset \bX$  denote the closure in $\bX$ of the Milnor set. This is of course an analytic set.

\begin{theorem}\label{t:atyppoint}\ 
Let  $f^{-1}(\lambda)$ be a fibre with a compact singular set.  Then:
\begin{enumerate}
 \rm \item \it  (V$_{\lambda}$)  $\Longleftrightarrow$  $\exists p\in L^{\ity}$ such that (V$_{(p,\lambda)}$).
  \rm \item \it  (S$_{\lambda}$)   $\Longleftrightarrow$  $\exists p\in L^{\ity}$ such that (S$_{(p,\lambda)}$).
    \rm \item \it  The set $\cA^{\ity}$ of atypical points at infinity  is contained in  the set $\cM := \overline{M(f)} \cap L^{\infty}\times \bR$, in particular it is finite.
  \end{enumerate}
  \end{theorem}
\begin{proof} We fix a large disk $D_{R}$ satisfying Lemma \ref{l:disk} and use the restriction $F = f_{|\bR^{2}\m \overline{D_R}}$, its fibres  $X_t =F^{-1}(t)$, their connected components $X_t^{j}$ and their closures $\overline{X_t^{j}}\subset \bX$.

\noindent (a). (V$_{(p,\lambda)}$) $\implies$ (V$_{\lambda}$)  follows directly from the Definitions \ref{d:split-van} and \ref{d:vanish}. Reciprocally, in case we have (V$_{\lambda}$), then  there exists a vanishing connected component at infinity $X_t^{j}$, and,  for $t\to\lambda$ and close enough to $\lambda$,  the limit set $\lim_{t\to \lambda}\overline{X_t^{j}}$ in $\bX$ is precisely one point. Indeed, note first that this limit is included in $\overline{F^{-1}(\lambda)} \cap L^{\ity}$ which is a finite set of points.  But if  that limit is not a single point, then we get  a contradiction  since the limit of a vanishing connected component must be connected too. 

So this limit point is of type  $(p,\lambda)\in L^{\infty}\times\bR$, and it is necessarily in $\overline{M(f)} \subset \bX$ since the component $X_t^{j}$ has a local minimum for the Euclidean distance function $\rho$, compare to Definition \ref{d:vanish}. Therefore, for $t$ close enough to $\lambda$,  the component $\overline{X_t^{j}}$ is a loop vanishing at $(p,\lambda)$.

\noindent (b).
If (S$_{(p,\lambda)}$) then $f$ has a splitting at infinity at $\lambda$, which follows again directly from the above Definitions \ref{d:split-van} and \ref{d:vanish}. Reciprocally, for a splitting component at infinity $X_t^{j}$, the 
point $\max_{q\in X_t^{j}\cap M(f)}\Vert q \Vert$
 is by definition in the Milnor set  $M(f)$ and therefore its limit for $t\to\lambda$ is 
a  point $(p,\lambda)\in L^{\infty}\times\bR$ at which $\overline{X_t^{j}}$ has a local representative which splits in the sense 
of Definition \ref{d:split-van}. 

\noindent (c).  That  $\cA^{\ity} \subset \cM$ follows directly from the proofs of (a) and (b). The finiteness of $\cA^{\ity}$ follows since the Milnor set  $M(f)$ is an affine curve, hence $\cM$ is a finite set.

This finishes the proof of our Theorem  \ref{t:main2}.
\end{proof}

We end  \S \ref{ss:proofmain} by remarking that all the results proved in this subsection concern only the behaviour ``at infinity'' of $f$ and therefore they are independent of the fact that we replace $f$ by some restriction $F = f_{|\bR^{2}\m \overline{D_R}}$.    



\section{Atypical points at infinity: selecting the points of $\cM$}

Theorem \ref{t:atyppoint}(c) tells that if $(p,\lambda)$ is an atypical point at infinity of $f$, then $(p,\lambda)$ belongs in the boundary at infinity $\cM = \overline{M(f)} \cap L^{\infty}\times \bR$ of the Milnor set $M(f)$.     Our first algorithm consists of determining the finite set $\cM$, see \S \ref{ss:M}. 
  
However the converse of the inclusion $\cA^{\ity} \subset \cM$ is not true: one can have polar branches at infinity without  vanishing or splitting, as shown  by an example  explained in \cite{TZ}, see next Example \ref{ex:king}.

\medskip
  
The remaining test: ``\emph{which point of the set $\cM$ is atypical and which is not}'' turns out to be delicate. We shall proceed in several steps.

The topological triviality at infinity of the function $f$ over a small enough interval $I\subset \bR$ which does not contain singular values of $f$ is well-understood,  see e.g. \cite{TZ, JoT},  \cite[\S 2]{Ti-book}.
This is however not enough for us here.  We need the localisation of this result at infinity, as follows. %

\subsection{Choice of the disk $D_{\e}$}\label{ss:disktheory}

Unless otherwise specified, we work here with a fibre  $f^{-1}(\lambda)$ having a compact singular set, and satisfying Lemma \ref{l:disk}.
Let us consider some point $(p,\lambda) \in \cM$ and a local chart $U = \bR^{2}$ of $\bP^2$ at $p$. We want to fix a ``box'' $D\times ]\lambda - \delta, \lambda+\delta[  \subset \bR^{2}\times \bR$ around $p$ where to apply some meaningful tests (which we shall define later) of local vanishing and splitting.  We start by determining the size of the box and insure the independency on choices.

\begin{theorem}\label{t:localfibration}
 Let  $\lambda \in \bR$, let $(p,\lambda) \in \cM$ and let $\bR^{2}$  be an affine chart of $\bP^{2}$ containing $p\in L^{\ity}$.  
There exists $\e_{0}>0$ such that  for any $0<\e <\e_{0}$ the map:  
\begin{equation}\label{eq:projfib}
  \tau_{|} : \bX \cap D_{\e}\times ]\lambda - \delta, \lambda[ \cup  ]\lambda, \lambda+\delta[   \longrightarrow   
]\lambda - \delta, \lambda[ \cup  ]\lambda, \lambda+\delta[
\end{equation}
  is a locally trivial fibration for any small enough $\delta$ depending on $\e$.
\end{theorem}
  \begin{proof}
 At each fixed point $(p,\lambda) \in \cM$ there exists  $\e_{0}>0$ such that the circles $\partial D_{\e}$ are transversal to $X_{\lambda} := \tau^{-1}(\lambda)$, and in particular $X_{\lambda}\cap D_{\e}$ is non-singular, for any $0< \e \le \e_{0}$.  This is a well-known  application of  the Curve Selection Lemma\footnote{See e.g. \cite{Mi}.} to the local distance function in the chart at infinity $U = \bR^{2}$ where $p$ is the origin $(0,0)$. 
 
 From the fact that $\overline{M(f)}$ is an analytic set and that $\overline{M(f)} \cap \overline{X_\lambda} \cap D_{\e} = \{p\}$, it follows that for any  fixed positive $\e<\e_{0}$,   there exists $\delta= \delta(\e)$ such that
$\overline{M(f)} \cap \bX \cap (\partial D_{\e}\times ]\lambda - \delta, \lambda+\delta[) =\emptyset$, which is again an application of the Curve Selection Lemma.

The choices of $\e_{0}$ and of $\delta(\e)$ are effective, and an algorithmic procedure shall be given in \S \ref{ss:disk}.
 

%

\begin{lemma}\label{p:ro-inf}
Let   $t\not \in f(\Sing f)$,   $(p,t)\not \in \cM$ and let $U=\bR^{2}$  be a chart of $\bP^{2}$ with  $p= (0,0)$.  
Let  $D\subset \bR^{2}$ be some disk at $p$ in this chart, such that the fibre $X_{t} := \tau^{-1}(t)$ is transversal to the  boundary $\partial D$.  There exists  $\eta >0$ such that:
 \begin{enumerate}
\rm  \item \it  $M(f) \cap \bX \cap D\times ]t - \eta, t +\eta[  \subset D_{R}$, for any affine disk  $D_{R} := \{\rho(x,y) = x^{2}+ y^{2}<R\} \subset \bR^{2}$ of large enough radius $R\gg 1$.
\rm  \item \it   the projection
\[ \tau_{|} : \bX \cap D\times ]t - \eta, t +\eta[   \longrightarrow   ]t - \eta, t +\eta[ \]
is a topologically trivial fibration.
\end{enumerate}
 \end{lemma}
\begin{proof} 
The transversality of the circle $\partial D$ to the curve fibre $X_{t}$ implies the existence of a  small enough $\eta >0$ such that $\partial D$ is still transversal to $X_{t'}$ for any $t'\in ]t - \eta, t +\eta[$. 
 The assumption $(p,t)\not \in \cM$ implies the existence of a large enough radius $R$ such that 
the complement of the affine disk $D_{R}$ does not intersect $M(f) \cap \overline  T_{\eta}$, where  $\overline  T_{\eta} :=\bX\cap \overline  D\times ]t - \eta, t +\eta[$ is what we call ``tube'' at $(p,t)$. \\ 
This proves (a).

\smallskip

\noindent 
(b). We fix $\eta >0$ like in (a). Since the Milnor set $M(f)$ does not intersect $\overline T_{\eta}\m D_{R}$, one may construct a \emph{$\rho$-controlled trivialisation}\footnote{As in \cite[Definition 3.1.8]{Ti-book}} on $\overline T^{*}_{\eta}\m D_{R}$, where $\overline T^{*}_{\eta} := \overline T^{*}_{\eta} \m (\{p\} \times ]t - \eta, t +\eta[)$.
This means a vector field tangent to the levels $\{\rho =r \}$ for $r\ge R$ which lifts $\partial/\partial s$ by $\tau_{|}$. 
The existence of such a lift of $\partial/\partial s$ is due to the transversality of the spheres $\{\rho =r \}$ to $X_{s}$ for all $r\ge R$ and all  $s \in ]t - \eta, t +\eta[$, conditions which are insured here by (a).
This vector field  produces a trivialisation $F: \overline T^{*}_\eta\m D_{R} \to  Y^{*}_t\times[t - \eta, t +\eta]$, where $Y^{*}_t := X_{t}\cap \overline T^{*}_\eta\m D_{R}$, and thus
the restriction $\tau_|: \overline  T^{*}_\eta\m D_{R} \to [t - \eta, t +\eta]$ becomes a trivial C$^{\ity}$ fibration.

This trivialisation has a unique extension to a C$^{0}$ trivialisation $\overline F: \overline T_\eta\m D_{R} \to Y_t \times[t - \eta, t +\eta]$ where $\overline F$ is  the identity on 
$\{p\}\times[t - \eta, t +\eta]$.

We thus get a topologically trivial fibration:
\begin{equation}\label{eq:fib1}
 \tau_{|} : T_{\eta}\m D_{R} \to ]t - \eta, t +\eta[.
\end{equation}

 We now have to extend this trivial fibration to the whole tube $T_{\eta}$. On the compact manifold with boundary $T_{\eta} \cap \overline{D_{R}}$, the map $\tau_{|}$ is a submersion and it is transversal to the boundary $T_{\eta}\cap \partial D_{R}$ by our transversality assumptions.
 We may then apply the Ehresmann fibration theorem and get a topologically trivial fibration:
\begin{equation}\label{eq:fib2}
 \tau_{|} : T_{\eta}\cap \overline{D_{R}} \to ]\lambda - \eta, \lambda +\eta[.
\end{equation}
By construction,  the two fibrations \eqref{eq:fib1} and \eqref{eq:fib2} glue together along their common boundary $T_{\eta}\cap \partial D_{R}$.	
\end{proof}

Applying now Lemma \ref{p:ro-inf} to $D = D_{\e}$
  and $t\in ]\lambda - \delta, \lambda[ \cup  ]\lambda, \lambda+\delta[$ shows that
\eqref{eq:projfib} is a locally trivial fibration.

Theorem \ref{t:localfibration} is proved.
\end{proof}

On the practical side, Theorem \ref{t:localfibration} tells of course that the map \eqref{eq:projfib} is topologically trivial on each interval  $]\lambda - \delta, \lambda[$ and   $]\lambda, \lambda+\delta[$, and moreover we have:

\begin{corollary}\label{c:splitvan}
Let  $(p,\lambda) \in \cM$ and let $U = \bR^{2}$ be an affine chart  of $\bP^{2}$ having $p$ as the origin.  
Let $\e_{0}>0$ such that the circles $\partial D_{\e}$ are transversal to $X_{\lambda}$ for any $0< \e \le \e_{0}$. 
Let  $0< \e \le \e_{0}$ and $\delta(\e)>0$ such that Theorem \ref{t:localfibration} holds.
 Then:  
\begin{enumerate}
\rm \item \it $X_{t_{0}}\cap  D_{\e}$ has a loop at 0 if and only if $X_{t}\cap  D_{\e}$ has a loop at 0,  for any $t$ in the same interval as $t_{0}$. In this case the condition $(V_{(p,\lambda)})$ holds.
\rm \item \it    $X_{t_{0}}\cap  D_{\e}$ has a connected component  not containing $p$ if and only if $X_{t}\cap  D_{\e}$ has a connected component  not containing $p$, for any $t$ in the same interval as $t_{0}$. 
 In this case the condition  $(S_{(p,\lambda)})$ holds.
 \end{enumerate}
\fin
\end{corollary} 
 
\subsection{Relation to King's notion of ``coalescing''}
The local topological triviality problem, a long-standing one, has been studied in different contexts and generality, see e.g. King's work \cite{Ki, Ki0}. Let us see what are the relations to our problem.

We are concerned here with a family of functions germs $G:(\bR^2\times I,  (0,0)\times I) \to (\bR,0)$, where $ G(y,z,t)=g_t(y,z)$ and $t\in I  := ] \lambda - \delta,  \lambda + \delta[$. We denote by $\Sing_{y,z} G$ the germ of the singular locus of $G$ as function of $x$ and $z$ and depending on the parameter $t$.  King \cite{Ki} defined the following notion associated to such a family of germs of continuous functions: 
		
\begin{definition}\label{d:coal} 
One says that the family $G$ has \emph{no coalescing} of critical points at the point $(0,0) \times \{\lambda\}$  if  $(0,0) \times \{\lambda\} \not \in \overline{\Sing_{y,z} G \m  ((0,0)\times I)}$.
		\end{definition}
		
	In \cite[Corollary 1]{Ki} it is proved that if the family $g_{t}=0$ has the same topological type for $t\in I$ for some small  $\delta$ and has no coalescing then there is a continuous family of homeomorphism germs $h_t:(\bR^2, 0) \to (\bR^2,0)$, for $t$ close enough to $\lambda$,  such  that $g_{\lambda}=g_t\circ h_t$. Hence King's reasoning starts once our problem is solved.   
  Moreover, the coalescing  can happen at some point $(p,\lambda)$ even if $\lambda$ is a typical value, and moreover, even if $(p,\lambda)\not\in \cM$.  Let us show this by the following two examples.  
		
\begin{example}\label{ex:King}(extracted from \cite{TZ})\label{ex:king}

 Consider $f:\bR^2 \to\bR$, $f(x,y) = 2x^2y^3-9xy^2+12y$. The point $p= [1:0:0]$ is the origin of the chart where we have  the family of functions induced by $f$: 
\[G(y,z,t)=g_{t}(y,z)= 2y^3-9y^2z^2 +12yz^4 -tz^5.  \] 
By the results of \cite{TZ}, there is no vanishing and no splitting at the fibre $f^{-1}(0)$, thus by our  Theorem \ref{t:atyppoint} or by  \cite{TZ, JoT}, this fibre is  not atypical.

On the other hand we have: 
			\[\frac{\partial G}{\partial y}(y,z,t) = 6y^2 -18yz^2 +12z^4,\ \
			\frac{\partial G}{\partial z}(y,z,t) = -18y^2z +48yz^3-5tz^4,\] 			
and the curve $\gamma(s)=(\frac{s^2}{36}, \frac{s}{6}, s)$ verifies  $\frac{\partial G}{\partial y}(\gamma(s))\equiv 0$,   $\frac{\partial G}{\partial z}(\gamma(s))\equiv 0$, thus $\im \gamma \subset \Sing_{y,z} G$.   Since  $G(\gamma(s))=\frac{-s^6}{6^6} \to 0$ for $s\to 0$, the family $G$ has coalescing at $(0,0)$. 
\end{example}		
\begin{example}[extracted from {\cite{DT}}]\label{ex:2} \

Let $f(x,y)= y(x^2y^2+3xy+3)$. We have $f(\Sing f)=\emptyset$,  $\cM= \emptyset$ and therefore $\cB_{f}=\emptyset$. Let $p= [1:0:0]$ be the origin of the chart $U$ of variables $(y, z)$. The local  family of functions induced by $f$ is: 
	\[G(y,z,t)=g_{t}(y,z)= y^3+3y^2z^2 +3yz^4 -tz^5.  \] 
Since  $\cM= \emptyset$, it follows from Lemma \ref{p:ro-inf} that the family of curve germs $ \{g_t=0\}$ is  topologically trivial at $(p,0)$. 
	
	On the other hand we have: 
	\[\frac{\partial G}{\partial y}(y,z,t) = 3y^2 +6yz^2 +3z^4,\ \ 
	\frac{\partial G}{\partial z}(y,z,t) = 6y^2z + 12 yz^3-5tz^4,\] 
and the curve $\gamma(s)=(-s^2, -s,\frac{6s}{5})$ verifies  $\frac{\partial G}{\partial y}(\gamma(s))\equiv 0$,   $\frac{\partial G}{\partial z}(\gamma(s))\equiv 0$. We have $\im \gamma \subset \Sing_{y,z} G$, and  $G(\gamma(s))=\frac{s^6}{5} \to 0$ for $s\to 0$ shows that the family  $G$ has coalescing at $(0,0)$. 

\end{example}

\bigskip



\section{Effective tests. The algorithms}\label{s:algo}
Here we describe effective algorithmic procedures which do the following:

\medskip

\noindent
\S \ref{ss:M}.  Find a finite set  of points $\cB$ containing  the  set $\cM$ of ``Milnor points at infinity'', where $\# \cB \le d = \deg f$.

\medskip
\noindent
\S \ref{ss:disk}.  For each such point $(p,\lambda)\in \cB$, find an effective tube neighbourhood $D_{\e}\times ]-\delta+\lambda, \delta + \lambda[$ used in Theorem \ref{t:localfibration}, and

\medskip
\noindent
\S \ref{ss:algovanish}.  Test the existence of vanishing or of splitting at some point $(p,\lambda)\in \cB$.
\\



\medskip

\subsection{Algorithm for estimating  $\cM$ by a finite set}\label{ss:M}

As before, we assume that $f$ is primitive, see after Definition \ref{d:milnor}. In order to find the finite set $\cM = \overline{M(f)} \cap L^{\infty}\times \bR$, we first determine the intersection of $\overline{M(f)}\subset \bP^{2}$ with the line at infinity $L^{\infty}$. This is easy: homogenise with the variable $z$ the equation $y\frac{\partial f}{\partial x}(x,y)-x\frac{\partial f}{\partial y}(x,y) =0$ of degree $d$, and then take $z=0$.  The number of real solutions $\overline{M(f)} \cap L^{\infty}$ will be at most $d$, by Bezout.

Next, for each point $p\in \overline{M(f)} \cap L^{\infty}$ we  determine the finite subset $\tau(\overline{M(f)}\cap\{p\}\times\bR)\subset \bR$. This is defined as follows:
\[ S_{p}(f)=\{t\in\bR\mid \exists\{(x_j,y_j)\}_{j\in\bN}\subset M(f), [x_j:y_j:1]\to p,\ f(x_j,y_j)\to t\}. \]

\medskip

We may assume (modulo a linear change of variables) that $p=[0:1:0]\in L^{\ity}$,  and we work in the chart $\{y\neq 0\}$ with coordinates $(x,z)$. If we set $\hat f(x,z) :=\tilde f(x,1,z)$ then we have the equality
$\lim_{ [x:y:1]\to p}f(x,y)=\lim_{z\to 0}\frac{\hat f(x,z)}{z^d}$. Suppose that $\hat h(x,z)=0$ is the equation that defines $\overline{M(f)}$ around $p$.
Finding $S_{p}(f)$ is equivalent to the following problem:

(*)  find all limits $\lim \frac{\hat f(x,z)}{z^d}$ for $(x,z)\to 0$ and $\hat h(x,z)=0$.

\medskip
\noindent
We consider $\hat h(x,z)$ and $\hat f(x,z)$ as holomorphic germs at the origin $(0,0)$ of the chart centred at $p$.
We pass to complex variables, and we find all the Puiseux roots of $\hat h(x,z)=0$.  They are finitely many, their number is not more than $d = \deg f$, not counting the conjugates.  The Newton-Puiseux algorithm exhibits indeed all the roots by starting 
from each edge of the Newton polygon and following all the possible choices of constants at each step.

We shall give  below an upper  bound for the number of steps and show that we can work with the truncated roots
in the limits problem (*). We might then select only the real limits, but still we might get more limit points than the set   $S_{p}(f)$ of real limits. Nevertheless our procedure yields a set $\cB\supset \cM$ of points 
 $(p,\lambda)\in \overline{M(f)} \cap L^{\infty}$ with  $\# \cB \le d$.

A Puiseux parametrisation of some root of $\hat h(x,z)=0$ looks like $z=T^n$, $x=\sum_{j\geq 1} \lambda_j T^j$, and 
 the series may be infinite even if $\hat h(x,z)$ is a polynomial.
However we only need  the first $dn$ terms of it, they are enough to compute the limit in (*) since we have the equality:
\begin{equation}\label{eq:lim}
 \lim_{T\to 0} \frac {\hat f(\sum_{j\geq 1} \lambda_j T^j, T^n)}{T^{dn}}=
\lim_{T\to 0}\frac{\hat f(\sum_{1\leq j\leq dn} \lambda_j T^j, T^n)}{T^{dn}}.
\end{equation}

In order to reach the exact value of the limit, we need to find the value of $n$, hence to give a bound for the number of steps in the Newton-Puiseux process as it is described e.g. in  \cite{Wa}. Let  $\hat h(x,z)$ be general in 
$x$ of order $k>0$ (i.e. $k$ is the lowest point on the $x$-axis 
in the Newton polygon of $\hat h$).
This iterative procedure gives the parametrisation $z=T^n,\ x=\sum_{i\geq 0}\lambda_iT^{m_i}$  
by producing step-by-step the constants  $\lambda_i$, the rational numbers $\frac{m_i}n$, and also the positive integers integers $k_i$, $p_i$, $q_i$ 
 such that $k_0=k$, $k_i\geq k_{i+1}$, where $p_i$ and $q_i$ are relatively prime, defined as follows.
Starting with $\hat h(x,z)=\sum c_{r,s}z^rx^s$,  at the step $i$ we get $\hat h^{(i)}(x,z)=\sum c^{(i)}_{r,s}z^rx^s$ and we consider the Newton polygon of $\hat h^{(i)}$ in the $(r,s)$-plane. 
The integer $k_i$ is defined as its lowest point on the $s$-axis. We then choose an edge of the Newton polygon.
Suppose that $(r_1,s_1)$ and $(r_2,s_2)$ are the two points at the end of this edge with  $r_1 < r_2$, $s_1 > s_2$. Then $p_i$ and $q_i$ are the unique 
relatively prime positive integers that satisfy $q_i(r_2-r_1)=p_i(s_1-s_2)$.

Then $n:=\prod_{i\geq 0}q_i$. By construction we have that $k_i\geq k_{i+1}$,  and  if $q_i=1$ then $k_{i+1}=k_i$, and  also $s_1\leq k_i$. 
Since $p_i$ and $q_i$ are relatively prime,  $q_i$ is a divisor of $s_1-s_2$.  In particular, $q_i\leq s_1\leq k_i$. 
Since $q_i\neq 1$ only if $k_i>k_{i+1}$, we deduce that $n\leq k!$, and since $\deg \hat h\leq d$, we have also that $n\leq d$ and therefore $n\leq\min\{k!,d\}$.

  

\medskip

The equality \eqref{eq:lim} implies that it suffices to run the Newton-Puiseux process only until the exponent $\frac{m_i}n$ becomes $d$. The sequence 
$\{m_i\}$ is strictly increasing and therefore after $d\cdot\min\{k!,d\}$ steps in the process we get that $m_i\geq dn$ and therefore we can compute the desired  limit \eqref{eq:lim} precisely.

The result  of this algorithmic procedure is a finite set of points $\cB\supset \cM$ with $\# \cB \le d$.

\subsection{Algorithm for finding a good disk at some point at infinity}\label{ss:disk} \ \\

 We use in particular the notations from \S\ref{ss:atyp}.
Let $(p,\lambda)\in \bX^{\infty}$ and consider an affine chart  with coordinates $(x,z)$, where  $p=[0:1:0]$ is the origin. We use in particular the notations from \S\ref{ss:atyp} We find  effectively the constants  $\e_0=\e_0(p,\lambda)>0$  and $\delta=\delta(\e_0)>0$ which fit in the local fibration Theorem \ref{t:localfibration}, namely such that: 

\smallskip
\noindent
(A).  for every  $0<\e\leq \e_0$, the circle $\partial D_{\e}\subset \bR$ intersects the curve germ  representative $\{g_\lambda=0\}\subset \bR^{2}$ transversally, and

\noindent
(B).  $\partial D_{\e}$ intersects $\{g_t=0\}$ transversally, for every $t\in]\lambda-\delta,\lambda+\delta[$. 

\medskip
Condition (A) amounts to finding the size of the Milnor disk, i.e. the distance:
\[ \Delta_{(p,\lambda)} := \dist \left(  (0,0), \left\{ z\frac{\partial g_\lambda}{\partial x}(x,z)-x\frac{\partial g_\lambda}
{\partial z}(x,z)=0\right\}  \cap \{g_\lambda=0\} \m \{(0,0)\} \right), \]
and we may then take  $\e_0 :=\Delta_{(p,\lambda)}/2$ if  $\Delta_{(p,\lambda)}>0$, and $\e_0 := 1/2$ if $\Delta_{(p,\lambda)} =0$.

\sloppy
After  fixing $\e_0$ we compute minimum $s >0$ of the function
$h(x,z):=g_\lambda^2+\left(z\frac{\partial g_\lambda}{\partial x}(x,z)-x\frac{\partial g_\lambda}
{\partial z}(x,z)\right)^2$  on $\partial D_{\e_{0}}$, thus have the inequality: 

\[ |g_\lambda|+ \left|z\frac{g_\lambda}{\partial x}(x,z)-x\frac{\partial g_\lambda}{\partial z}(x,z)\right|\geq\sqrt{s}.\]

We then have the following inequalities for  $(x,z)\in \partial D_{\e_{0}}$: 
\[ |g_t|+\left|z\frac{\partial g_t}{\partial x}-x\frac{\partial g_t}{\partial z}\right|=
|g_\lambda-(t-\lambda) z^d|+\left|z\frac{g_\lambda}{\partial x}-x\frac{\partial g_\lambda}{\partial z}+
dx(t-\lambda) z^{d-1}\right|  \geq \]
\[\geq |g_\lambda|+\left| z\frac{g_\lambda}{\partial x}(x,z)-x\frac{\partial g_\lambda}{\partial z}(x,z)\right| -
|(t-\lambda) z^d| - |dx(t-\lambda) z^{d-1}| \geq \sqrt s - |t-\lambda|(\e_{0}^d+d\e_{0}^d).\]
We need the length $|t-\lambda|$ such that the last term is positive.
Then $\delta(\e_{0}) :=\frac{\sqrt s}{\e_{0}^d(d+1)}$ answers to condition (B).

\

Here is the sequence of tests:\\

$\bullet$ Solve $\left\{ g_\lambda=0 ; 
z\frac{\partial g_\lambda}{\partial x}(x,z)-x\frac{\partial (g_\lambda)}{\partial z}(x,z)=0 \right\}$. 
We obtain the finite set $A_{(p,\lambda)}$.
\medskip 

$\bullet$ $\e_{0}:=\frac 12\min\{\|(x,z)\| \mid (x,z)\in A_{(p,\lambda)}\}$.
\medskip

$\bullet$ Solve $\left\{ z\frac{\partial h}{\partial x}(x,z)-x\frac{\partial h}{\partial z}(x,z)=0, x^2+z^2=\e_{0}^2
\right\}$, where \\ $h(x,z):=g^2_\lambda(x,z)+\left(z\frac{\partial g_\lambda}{\partial x}(x,z)-x\frac{\partial g_\lambda}
{\partial z}(x,z)\right)^2$. \
We obtain a finite set $B_{(p,\lambda)}$.
\medskip

$\bullet$ $\delta(\e_{0}):=\frac{\sqrt{\min\{h(x,z) \  \mid \ (x,z)\in B_{(p,\lambda)}\}} }{\e_{0}^d(d+1)}.$

\medskip

\subsection{Algorithm to determine   the existence of local vanishing and splitting  }\label{ss:algovanish} \

From the above algorithms we obtain  the following situation: we have got a finite set $\cB\supset \cM$ of points at infinity. At each such point $(p,\lambda)\in \cB$, in some affine chart $\bR^{2}$ with origin at $p$,
 we have defined a ``twin-box'' $D_{\e}\times ]\lambda - \delta, \lambda[\cup ]\lambda,  \lambda+  \delta[$, with effective constants $\e>0$ and $0<\delta\ll \e$.  After Corollary \ref{c:splitvan}, we may now choose any value $t$ in $]\lambda - \delta, \lambda[$,  or in $]\lambda,  \lambda+  \delta[$, respectively, and thus fix a disk $D := D_{\e}\times \{ t\}$ in which to test the existence of the phenomena of vanishing (V$_{(p,\lambda)}$) and splitting (S$_{(p,\lambda)}$) for the affine curve $X_{t}= \overline{f^{-1}(t)}$.
 
Our  test follows closely the nice algorithm given in \cite[Section 11.6]{BPR} which fits perfectly to our problem.

Let us denote  $\cC := X_{t}\subset \bR^2$.
From the previous algorithm (and from Theorem \ref{t:localfibration})  we have that $\cC \cap D$ is non-singular outside the origin and that $\cC \pitchfork \partial D$. Determining if there is vanishing or splitting at $0$, respectively,  means  to determine whether: \\

\noindent
(a).  $\cC \cap D$ has a loop passing through $0$, or \\
(b). $\cC \cap D$ has a connected component which does not pass through $0$.

\

One defines a multigraph $\Gamma$ (i.e.  we allow  more than one edge between two vertices) homeomorphic to $\cC\cap \overline D$ as follows. 

We consider the projection $p_x:\bR^2\to\bR$, $(x,y)\mapsto x$ and its restriction $\pi := p_{x | \cC}$.

We may assume that no irreducible component of $\cC$ is  vertical,  and we shall handle at the end the special case of one or more vertical components. 

\smallskip
\noindent
\textbf{The set of vertices $V(\Gamma)$.} Let $W$ denote the set of points consisting of the origin $0\in D$, all the points of local maxima and of local minima of the projection $\pi$, and the points of intersection $\cC\cap \partial D$. 

The set of vertices of our graph is defined as:
\[ V(\Gamma):= \cup_{u\in \pi(W)}\pi^{-1}(u). \]

The origin is a distinguished point. The  vertices of $\cC\cap\partial D$ are also special and we may attach them a certain colour. These will   play a special role  at the end of the algorithm.

\smallskip
\noindent
\textbf{The set of edges $E(\Gamma)$.}

For each vertex $q\in V$, let $L(q)$ be the number of real curve arcs of $\cC$ in $\overline D$ that have $q$ 
as an end-point,  and that are projected by $\pi$ to the left of  $\pi(q)$. Similarly,
let $R(q)$ be the number of real curve arcs of $\cC$ in $\overline D$ that have $q$
as an end point, and that are projected by $\pi$ to the right of 
$\pi(q)$. 

Note that for a point $q$ in $V(\Gamma)\cap D\setminus \{0\}$ we have $L(q)+R(q)=2$, whereas 
for $q\in\cC\cap\partial D$ we have $L(q)+R(q)=1$.

Let $\pi(V(\Gamma))$ be the ordered set $a_1< \cdots <a_k$. For each point $a_j$,  let 
$\pi^{-1}(a_j)$ be the ordered set  $q_j^1>q_j^2>\cdots$.  

For each  $j$, starting from 1 and until $k$ increasingly, we consider the set of points
$q_j$ starting with  $q_j^1$, then $q_j^2$ and so on decreasingly,  and do the following step:

(*) if  $R(q_j^k)>0$ then we draw an edge
between $q_j^k$ and the first vertex  $q_{j+1}^s$ (i.e. the first index $s$ starting from 1)  with the property $L(q_{j+1}^s)>0$. 
  Decrease $R(q_j^k)$ and $L(q_{j+1}^s)$ by 1.
 
 \smallskip
 
After a certain number of steps, actually $\le 3d^2$,
all the values $R(q)$ and $L(q)$ decrease to 0, and the algorithm stops. The result  is our set of edges $E(\Gamma)$, hence our multigraph $\Gamma$.

\smallskip

\smallskip
\noindent
\textbf{Testing the  multigraph $\Gamma$.}

Our initial problem amounts now to check (by a standard algorithm) whether:

\noindent
(a'). $\Gamma$ has a cycle containing the origin $0\in D$, which corresponds to (a) above. This  means that there is a vanishing (V$_{(p,\lambda)}$),  or \\
(b').  there is path between two vertices of $\cC\cap\partial D$ which does not pass through $0$. 
This corresponds to (b) above, which  means that there is a splitting (S$_{(p,\lambda)}$).

Coming back to the possible vertical components:  if there are vertical components, then we only have to check if they pass through the origin or not, since in the former case the component is neutral for the above algorithm, and in the later case, it produces a splitting (S$_{(p,\lambda)}$).





 \end{document}